\documentclass[letterpaper, 12pt]{article}
\usepackage{amsmath,amsthm,amssymb,fullpage,tikz,subfig}
\usepackage{hyperref}
\usepackage{breakurl}
\usetikzlibrary{calc}

\newcommand{\F}{\mathcal{F}}
\newcommand{\Q}{\mathcal{Q}}
\newcommand{\E}{\mathbb{E}}		
\newcommand{\R}{\mathbb{R}}
\renewcommand{\d}{\rho}
\newcommand{\ex}{\mathrm{ex}}
\newcommand{\cH}{\mathcal{H}}
\newcommand{\cg}{cube graph}	

\newcommand{\D}{\displaystyle}
\newtheorem{theorem}{Theorem}
\newtheorem{lemma}{Lemma}

\newcommand\oururl{\url{http://www.math.uiuc.edu/~jobal/cikk/hypercube}}

\newcommand\middleresult{2.15121}

\title{Upper bounds on the size of $4$- and $6$-cycle-free subgraphs of the hypercube}
\author{
 J\'{o}zsef Balogh\thanks{ University of Illinois, Urbana-Champaign, {\tt jobal@math.uiuc.edu}.
    This material is based upon work supported in part by NSF CAREER
Grant DMS-0745185, UIUC Campus Research Board Grant 11067, and OTKA Grant
K76099.} 
  \and Ping Hu \footnote{University of Illinois, Urbana-Champaign, {\tt pinghu1@math.uiuc.edu}.}
  \and Bernard  Lidick\'{y} \footnote{University of Illinois, Urbana-Champaign, 
  Charles University, Prague,
   {\tt lidicky@illinois.edu}.}
 \and Hong Liu
\thanks{University of Illinois, Urbana-Champaign,
  {\tt hliu36@illinois.edu}.}   
 }
\date{\today}

\begin{document}

\maketitle

\begin{abstract}
In this paper we modify slightly Razborov's flag algebra machinery
to be suitable for the hypercube.
We use this modified method to show that the maximum number of edges of a 
$4$-cycle-free subgraph of the $n$-dimensional hypercube is at most $0.6068$ times the number
of its edges. 
We also improve the upper bound on the number of edges
for $6$-cycle-free subgraphs of the $n$-dimensional hypercube from $\sqrt{2}-1$ to  
$0.3755$ times the number of its edges.
Additionally, we show that if the $n$-dimensional hypercube is considered
as a poset, then the maximum vertex density of three middle layers
in an induced subgraph without 4-cycles is at most
$2.15121\binom{n}{\lfloor n/2 \rfloor}$.
\end{abstract}

\section{Introduction}

Let $\Q_n$ be the graph of the $n$-dimensional hypercube ($n$-cube) whose vertex set is the set $\left\{ 0,1 \right\}^n$ of binary $n$-tuples, and two vertices are adjacent if and only if they differ in exactly one 
coordinate. The \emph{Hamming distance} between two $n$-tuples $u$ and $v$, denoted by $d(u,v)$, is the number of coordinates in which they differ. 
So $uv$ is an edge  of $\Q_n$ if and only if $d(u,v) = 1$. 
Note that the hypercube $\Q_n$ has $2^n$ vertices and $n2^{n-1}$ edges.



Let $e(G)$ denote the number of edges of a graph $G$. For a graph $F$, we
define $\ex_{\Q}(n,F)$ to be the maximum number of edges of an $F$-free subgraph of $\Q_n$
and define 
 $$\pi_{\Q}(F) = \D\lim_{n\to\infty}\frac{\ex_{\Q}(n,F)}{e(\Q_n)}.$$
Note that the existence of the limit follows from an easy averaging argument that 
${\ex_{\Q}(n,F)}/{e(\Q_n)}$ is non-increasing as $n$ increases. 


Erd\H{o}s~\cite{Erdos:1984,Erdos:1990}  was the first one  who considered
Tur\'{a}n type problems for the hypercube. He proposed a problem of determining
$\ex_{\Q}(n,C_{2t})$, suggesting that for all $t > 2$ perhaps $o(e(\Q_n))$ was
an upper bound.
It turned out to be false for $t=3$ as Chung~\cite{Chung:1992} and Brouwer, 
Dejter and Thomassen~\cite{Brouwer:1993} found a $4$-coloring of the hypercube without a monochromatic $C_6$.  
This was later improved by Conder~\cite{Conder:1993} 
to a $3$-coloring. This implies that $\ex_{\Q}(n,C_{6}) \geq \frac{1}{3}e(\Q_n)$. On the other hand, the best known
upper bound obtained by Chung~\cite{Chung:1992} is $\ex_{\Q}(n,C_{6}) \leq (\sqrt{2}-1+o(1))e(\Q_n)$.

Chung~\cite{Chung:1992} also showed that Erd\H{o}s was right for even $t \geq 4$ by proving
that $\ex_{\Q}(n,C_{2t}) = o(e(\Q_n))$. F\"uredi and \"Ozkahya~\cite{Furedi:2009b, Furedi:2009} complemented the
previous result by showing $\ex_{\Q}(n,C_{2t}) = o(e(\Q_n))$ for all odd $t \geq 7$.
Their approaches were recently unified by Conlon~\cite{Conlon:2010}.
Despite the efforts in~\cite{Alon:2006,Axenovich:2006,Conlon:2010} the case $\ex_{\Q}(n,C_{10})$ still
remains unsolved.

Erd\H{o}s~\cite{Erdos:1984} was particularly interested in $\ex_{\Q}(n,C_{4})$.
He conjectured that the answer is $\pi_{\Q}(C_4) = 1/2$ and offered \$100 for a solution. 
Best known lower bound $\frac{1}{2}(1 + \frac{1}{\sqrt{n}})e(\Q_n)$ (valid when $n$ is a power of 4) 
on $\ex_{\Q}(n,C_4)$
was obtained by Brass, Harborth and Nienborg~\cite{Brass:1995}. 
The upper bound on $\pi_{\Q}(C_4)$ of $0.62284$ obtained by Chung~\cite{Chung:1992} was
recently improved by Thomason and Wagner~\cite{Thomason:2009} by a computer
assisted proof to $0.62256$. They also claimed that  $\pi_{\Q}(C_4) \le 0.62083$ 
can be obtained with the same technique.

Razborov~\cite{Razborov:2007} developed a systematic approach to bound 
densities of subgraphs called flag algebra. 
This method can be applied to various problems~\cite{Grzesik:2011,Hatami:2011,Hladky:2009,Kral:2011}.
One nice exposition of applying the method to Tur\'{a}n density is in \cite{Baber:2011},
for a recent development see~\cite{Falgas:2011}. 
We present a modification of the method for subgraphs of the hypercube. 
By applying our modified flag algebra method we obtained
improvements on the upper bounds on $\pi_{\Q}(C_4)$ and $\pi_{\Q}(C_6)$.




\begin{theorem}\label{thm:main}
$\pi_{\Q}(C_4) \le 0.6068$.
\end{theorem}

\begin{theorem}\label{thm:C6}
$\pi_{\Q}(C_6) \le 0.3755$.
\end{theorem}

These results were independently proved by Baber~\cite{Baber:2012} which originally appeared in his PhD thesis in March 2011. Baber also estimated vertex Tur\'{a}n density of $\Q_3$ and determined vertex Tur\'{a}n density of $\Q_3$ with one vertex removed for hypercubes.

Both proofs are computer assisted as the number of considered cases
is too large to be computed by hand without an extreme suffering (of students and a postdoc).
All the programs as well as their inputs and outputs can be obtained at
\oururl.

In addition to spanning subgraphs of the hypercube, flag algebras can be used
also for induced subgraphs of the hypercube. However, we present the result
in a lattice settings because of its original motivation.
For a family $F$ of subsets of $[n] = \{1, 2, \ldots, n\}$ ordered by inclusion, and a partially ordered set $P$, we say that $F$ is $P$-free if it does not contain a subposet isomorphic to $P$. Let $\ex(n, P)$ be the largest size of a $P$-free family of subsets of $[n]$. 
Let $Q_2$ be the poset with distinct elements $a, b, c, d$ where $a < b, c < d$; i.e., the $2$-dimensional Boolean lattice.
Axenovich, Manske and Martin \cite{Axenovich:2011} showed that $2N - o(N)
\le \ex(n, Q_2 ) \le 2.283261N + o(N)$ where $N = \binom{n}{\lfloor n/2\rfloor}$.
They also proved that the largest $Q_2$-free family of subsets of $[n]$ having
at most three different sizes has at most $(3+\sqrt{2})N/2$ members. 
Their latter result can be improved by using flag algebras. We show how to
achieve the same bound $(3+\sqrt{2})N/2$ that can be verified by hand. With help
of computers we then improve the bound to $\middleresult N$.

\begin{theorem}\label{thm:middle}
The largest $Q_2$-free family of subsets of $[n]$ having at most three different
sizes has at most $\middleresult N$ members
where $N = \binom{n}{\lfloor n/2\rfloor}$.
\end{theorem}

In the next section we give a brief introduction to the flag algebra method
and describe our modification of it to subgraphs of the hypercube. 
We refer the interested reader to the seminal paper of Razborov~\cite{Razborov:2007} for a detailed exposition of the method. In Section~\ref{sec:example} we apply the method with a simple setting and obtain an upper bound $\pi_{\Q}(C_4) \le 2/3$. 
The main purpose of Section~\ref{sec:example} is to make the reader comfortable 
with the terminology and describe the proof technique.
Finally, in Sections~\ref{sec:proof} and ~\ref{sec:proofC6} we give ideas of the proofs of Theorems~\ref{thm:main} and~\ref{thm:C6}, respectively. We do not include
all the technicalities of the proofs as the number of considered graphs is too large.
The interested reader may see all the technical details at \oururl.
The last section is devoted to giving a proof idea of Theorem~\ref{thm:middle}.
\section{The flag algebra method for the hypercube}
In this section we give a brief introduction to the flag algebra method mixed with the 
necessary modifications for subgraphs of the hypercube.
We say that a graph $G$ is a \emph{\cg} if $G$ is a subgraph of $\Q_n$ for some $n$, so $V(G)\subseteq \left\{ 0,1 \right\}^n$ and if $uv$ is an edge of $G$ then $d(u,v) = 1$. 

Given a {\cg} $G$ and a subset $U$ of $V(G)$, we denote the subgraph of $G$ induced by $U$ by $G[U]$.
It is easy to see that $G[U]$ is also a \cg.

Given a subset $U$ of $\left\{ 0,1 \right\}^n$, let $D(U)$ be the set of
coordinates $i$ such that there exist $v,w\in U$ which differ in the coordinate $i$
($v$ and $w$ may differ in more coordinates). 
If $U=\left\{u,v\right\}$, then we abbreviate $D(\left\{u,v\right\})$ to $D(u,v)$. Let $d(U) = |D(U)|$ and 
again $d(\left\{u,v\right\})$ is abbreviated to $d(u,v)$, as it is the Hamming distance of $u$ and $v$.
We define the \emph{dimension} of a {\cg} $G$ to be $\dim(G) = d(V(G))$. 
Given a vertex $v\in \{0,1\}^n$, let $v[i]$ be its $i^{\text{th}}$ coordinate.
Given a vertex set $U\subseteq \{0,1\}^n$ of dimension $r$, let $Q(U)$ be the
set of vertices of the unique $r$-cube containing $U$, i.e.
$$Q(U)=\left\{ v:v\in \{0,1\}^n,\forall u\in U,i\notin D(U),v[i]=u[i] \right\}.$$

Given $V\subseteq \left\{ 0,1 \right\}^m$ and $U\subseteq \left\{ 0,1 \right\}^n$, we say a map $f: V\rightarrow U $ is \emph{Hamming distance preserving} if $\forall u,v\in V, d(u,v) = d(f(u),f(v))$. Note that a Hamming
distance preserving map is injective since $d(u,v) = 0 \textrm{ iff } u=v$. When $U=V=\left\{ 0,1 \right\}^n$, such $f$ is a \emph{cube automorphism}.
We call a map $f: V\rightarrow U$ \emph{feasible} if there exists a Hamming 
 distance preserving map $\tilde{f}: Q(V)\rightarrow Q(U) $ such that 
 $f(v) = \tilde{f}(v)$ for all $v \in V$.
Given two {\cg}s $H$ and $G$, we say $H$ and $G$ are \emph{feasible isomorphic} 
(denoted by $H \simeq G$) if there exists a feasible bijection 
$f: V(H)\rightarrow V(G)$ satisfying $\forall u,v\in V(H), f(u)f(v) \in E(G)\textrm{ iff } uv\in E(H)$. Such $f$ is called a \emph{feasible isomorphism} from $H$ to $G$. 
See Figure~\ref{fig-cubes} for an example.

\begin{figure}[htbp]
\begin{center}
\def\e{1.3}
\def\lw{1.1pt}

\def\cube
{
\draw[dotted]
\foreach \x in {0,1}
{
(0.5*\x*\e,0.5*\x*\e)  node[noin] (a1\x) {} 
--  ++(\e,0) node[noin] (a2\x) {}
--  ++(0,\e) node[noin] (a3\x) {}
--  ++(-\e,0) node[noin] (a4\x) {}
-- (a1\x)
}
\foreach \y in {1,2,3,4}
{
(a\y0) -- (a\y1)
}
(a10) node[insep] {}
(a20) node[insep] {}
;
}

\def\confA{
\cube
\draw[line width=\lw]
(a10) -- (a20) -- (a30) -- (a40)
(a30) node[insep] {}
(a40) node[insep] {}
(a41) node[insep] {}
(a21) node[insep] {}
;
\draw (0.5*\e,-0.6) node {$G_1$};
}

\def\confB{
\cube
\draw[line width=\lw]
(a10) -- (a11) -- (a41) -- (a40)
(a11) node[insep] {}
(a40) node[insep] {}
(a41) node[insep] {}
(a31) node[insep] {}
;
\draw (0.5*\e,-0.6) node {$G_2$};
}

\def\confC{
\cube
\draw[line width=\lw]
(a10) -- (a20) -- (a30) -- (a31)
(a30) node[insep] {}
(a31) node[insep] {}
(a41) node[insep] {}
(a21) node[insep] {}
;
\draw (0.5*\e,-0.6) node {$G_3$};
}

\begin{tikzpicture}
[insep/.style={inner sep=1.7pt, outer sep=0pt, circle, fill}, 
noin/.style={inner sep=0pt, outer sep=0pt, circle, fill},
list/.style={inner sep=0.0pt, outer sep=0pt, rectangle,fill=white},
precol/.style={inner sep=1.8pt, outer sep=0pt, circle, fill},
three/.style={inner sep=1.5pt, outer sep=0pt, regular polygon,regular polygon sides=3, draw,fill=white},
four/.style={inner sep=1.8pt, outer sep=0pt, regular polygon,regular polygon sides=4, draw,fill=white},
five/.style={inner sep=1.8pt, outer sep=0pt, regular polygon,regular polygon sides=5, draw,fill=white}
]

\begin{scope}[xshift=0cm, yshift = 0cm] 
\confA 
\end{scope}

\begin{scope}[xshift=3.5cm, yshift = 0cm] 
\confB
\end{scope}

\begin{scope}[xshift=7cm, yshift = 0cm] 
\confC
\end{scope}

\end{tikzpicture}
\end{center}
\caption{All $G_1,G_2$ and $G_3$ are isomorphic. However, only  $G_1$ and $G_2$ are feasible isomorphic.}
\label{fig-cubes}
\end{figure}
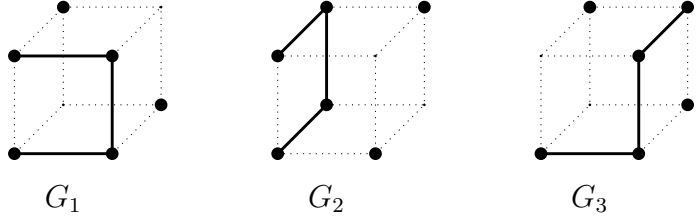

It is not hard to see that a feasible map preserves the dimension. Indeed, we have a stronger statement.

\begin{lemma} \label{Dmap}
Let $V\subseteq \{0,1\}^m, U\subseteq \{0,1\}^n$ and let $f:V \rightarrow U$ be a feasible map. 
Then there exists an injective map $\phi:D(V)\rightarrow D(U)$ such that for any subset $V'\subseteq V$, we have $D(f(V')) = \phi(D(V'))$. Given $\phi$ and $f(v)$ for any $v\in V$, then $f$ is uniquely determined.
\end{lemma}
\begin{proof}
As $f$ is feasible, there exists a Hamming distance preserving map $\tilde{f}: Q(V)\rightarrow Q(U)$ 
such that $f(v) = \tilde{f}(v)$ for every $v\in V$.
We start by inspecting $\tilde{f}$. Let $d(V) = k$ and 
$D(V)=\{l_1,\dots,l_k\}$. Pick a vertex $v\in V$ and let $v_{i}\in Q(V)$ be the vertex 
which differs from $v$ only in the coordinate $l_{i}$. As $\tilde{f}$ is Hamming distance 
preserving, $\tilde{f}(v_i)$ differs from $\tilde{f}(v)$ in only one coordinate, say  
$l_i'$. Then we have $l_i'\ne l_j'$ for $i\ne j$ since $\tilde{f}$ is injective. 
Next we define $\phi(l_i) = l_i'$ for all $1 \leq i \leq k$ and show that it satisfies our needs.
Because $\tilde{f}$ is Hamming distance preserving, for a vertex $u\in Q(V)$ we have 
$D(\tilde{f}(u),\tilde{f}(v))=\phi(D(u,v))$, which means $f$ is uniquely determined by $\phi$ and $f(v)$.
Furthermore, for any two vertices $v_1,v_2\in Q(V)$ we have 
$D(\tilde{f}(v_1),\tilde{f}(v_2))=\phi(D(v_1,v_2))$ since 
$$D(\tilde{f}(v_1),\tilde{f}(v_2))=D(\tilde{f}(v),\tilde{f}(v_1))\bigtriangleup D(\tilde{f}(v),\tilde{f}(v_2))$$
and $\phi(D(v_1,v_2))=\phi(D(v,v_1))\bigtriangleup\phi(D(v,v_2))$,
where $\bigtriangleup$ means the symmetric difference of the sets. 
Then for any subset $V'\subseteq V$, we have $D(f(V')) = \phi(D(V'))$.
\end{proof}


Let $F$ be a fixed graph.
Our goal is to compute an upper bound on $\pi_{\Q}(F)$. Let $\mathcal{H}_s$ be the family of all $F$-free spanning subgraphs of $\Q_s$, up to cube automorphism.

Given any two {\cg}s $H$ and $G$, we define $p(H,G)$ to be the probability that a 
feasible map $f:V(H)\rightarrow V(G)$ chosen uniformly at random satisfies $G[Im(f)]\simeq H$.
Note that if $H\in \mathcal{H}_s$ and $V(G) = V(\Q_n)$ then $\Q_n[Im(f)]\simeq \Q_s$. 

Given a {\cg} $G$, let $n=\dim(G)$, then define its edge density $\d(G) = e(G)/e(\Q_n)$.
Let $G$ be an $F$-free spanning subgraph of $\Q_n$. By averaging over all $H \in \mathcal{H}_s$ we have 
\begin{equation*}
\d(G) = \sum_{ H \in \mathcal{H}_s}\d(H)p(H,G)
\end{equation*}
as $\sum_{ H \in \mathcal{H}_s}p(H,G) = 1$. Hence $\d(G) \le \max_{ H \in \mathcal{H}_s}\d(H)$ and then $\pi_{\Q}(F)\le \max_{ H \in \mathcal{H}_s}\d(H)$. 

This bound in general is very poor, for $F=C_4$ and $s\in\{2,3,4\}$ it gives that $\pi_{\Q}(F) \le 3/4$. 
It is because this bound only considers $\d(H)$. It does not use other structural properties of graphs in $\mathcal{H}_s$. 
Razborov's flag algebra method allows us to make use of more information about $\mathcal{H}_s$ 
and hence it gives a much better bound. 
Indeed, our results are obtained with $s=3$.

Let $H$ be a {\cg}, we call an injective map $\theta:
[m]\rightarrow V(H)$ a \emph{type map to $H$} if every vertex
$v \in V(H) \setminus Im(\theta)$ satisfies $v \notin Q(Im(\theta))$. 
A \emph{flag} $(H,\theta)$ is $H$ together with a type map $\theta$. If $\theta$ 
is also bijective, then we call the flag a \emph{type}. We can think of $\theta$ as a labeling. If $m=0$, then no vertex is labeled, and we use $0$ to denote such type.
Let $F_1=(H,\theta)$ be a flag. 
We say $F_1$ is \emph{$F$-free} if $H$ is $F$-free. 
We say $F_1$ is a \emph{$\sigma$-flag} if $(Im(\theta),\theta) \simeq \sigma$. 
See Figure~\ref{fig-okflag} for examples.
Let $H_1,H_2$ be two {\cg}s. We call two flags $F_1 = (H_1,\theta_1)$ and $F_2=(H_2,\theta_2)$
\emph{isomorphic} (denoted by $F_1 \simeq F_2$) if there
exists a feasible isomorphism $f:V(H_1)\rightarrow V(H_2)$ satisfying $f\cdot\theta_1 = \theta_2$. 

\begin{figure}[htbp]
\begin{center}
\def\e{1.3}
\def\lw{1.1pt}
\def\sh{0.18}

\def\type
{
\draw[fill=gray!30,rounded corners=4pt]
(-\sh,0.5*\e) -- ++(0,0.5*\e+\sh) -- ++(\e+2*\sh, 0) -- ++(0,-2*\sh) 
-- ++(-\e,0) -- ++(0,-\e) -- ++(-2*\sh,0) -- cycle
;
\draw[dotted]
(0,0) 
 node[insep, label={left:3}] {}
-- ++(\e,0) -- ++(0,\e)
(0,0) -- ++(0,\e)
;
\draw[line width=\lw]
(0,\e)  node[insep, label={left:1}] {} 
-- ++(\e,0)  node[insep, label={right:2}] {}
;
}

\def\bad
{
\draw[line width=4pt, color=gray!20, dashed]
(0+0.1,0-0.5) -- (1.5*\e-0.1,1.5*\e+0.5)
(1.5*\e-0.1,0-0.5) -- (0+0.1,1.5*\e+0.5)
;
}

\def\cube
{
\draw[dotted]
\foreach \x in {0,1}
{
(0.5*\x*\e,0.5*\x*\e)  node[noin] (a1\x) {} 
--  ++(\e,0) node[noin] (a2\x) {}
--  ++(0,\e) node[noin] (a3\x) {}
--  ++(-\e,0) node[noin] (a4\x) {}
-- (a1\x)
}
\foreach \y in {1,2,3,4}
{
(a\y0) -- (a\y1)
}
(a41) node[insep] {}
;
\draw[line width=\lw]
(a40)--(a41)
;
}

\def\confA{
\type
\cube
\draw[line width=\lw]
(a41)--(a31)--(a30)
(a31) node[insep] {}
;
\draw (0.5*\e,-0.6) node {$F_1$};
}

\def\confB{
\type
\cube
\draw[line width=\lw]
(a41)--(a31)--(a30)
(a21)--(a31)
(a31) node[insep] {}
(a21) node[insep] {}
;
\draw (0.5*\e,-0.6) node {$F_2$};
}

\def\confC{
\type
\cube
\draw[line width=\lw]
;
\draw (0.5*\e,-0.6) node {$F_1$};
}

\def\confD{
\bad
\type
\cube
\draw[line width=\lw]
(a10)--(a20)
(a20) node[insep] {}
;
\draw (0.5*\e,-0.6) node {$F_3$};
}

\def\confType{
\type
\draw (0.5*\e,-0.6) node {$\sigma$};
}

\def\confE{
\draw[fill=gray!30,rounded corners=4pt]

(-\sh,\e-\sh) rectangle (\e+\sh,\e+\sh)
;
\draw[dotted]
(0,0) 
 node[insep] {}
-- ++(\e,0) -- ++(0,\e)
(0,0) -- ++(0,\e)
;
\draw[line width=\lw]
(0,\e)  node[insep, label={left:1}] {} 
-- ++(\e,0)  node[insep, label={right:2}] {}
;
\cube
\draw[line width=\lw]
(a10)--(a20)
(a20) node[insep] {}
;
\draw (0.5*\e,-0.6) node {$F_4$};
}

\begin{tikzpicture}
[insep/.style={inner sep=1.7pt, outer sep=0pt, circle, fill}, 
noin/.style={inner sep=0pt, outer sep=0pt, circle, fill},
list/.style={inner sep=0.0pt, outer sep=0pt, rectangle,fill=white},
precol/.style={inner sep=1.8pt, outer sep=0pt, circle, fill},
three/.style={inner sep=1.5pt, outer sep=0pt, regular polygon,regular polygon sides=3, draw,fill=white},
four/.style={inner sep=1.8pt, outer sep=0pt, regular polygon,regular polygon sides=4, draw,fill=white},
five/.style={inner sep=1.8pt, outer sep=0pt, regular polygon,regular polygon sides=5, draw,fill=white}
]

\begin{scope}[xshift=0cm, yshift = 0cm] 
\end{scope}

\begin{scope}[xshift=3.5cm, yshift = 0cm] 
\confC
\end{scope}

\begin{scope}[xshift=7cm, yshift = 0cm] 
\confB
\end{scope}

\begin{scope}[xshift=10.5cm, yshift = 0cm] 
\confD
\end{scope}

\begin{scope}[xshift=14cm, yshift = 0cm] 
\confE
\end{scope}

\begin{scope}[xshift=0.25cm, yshift = 0cm] 
\confType
\end{scope}

\end{tikzpicture}
\end{center}
\caption{$\sigma$ is a type, $F_1$ and $F_2$ are $\sigma$-flags but $F_3$ is not a flag. It contains an unlabeled vertex in $Q(Im(\theta))$. $F_4$ is a flag but not a $\sigma$-flag as the labeled vertices do not induce $\sigma$.}
\label{fig-okflag}
\end{figure}
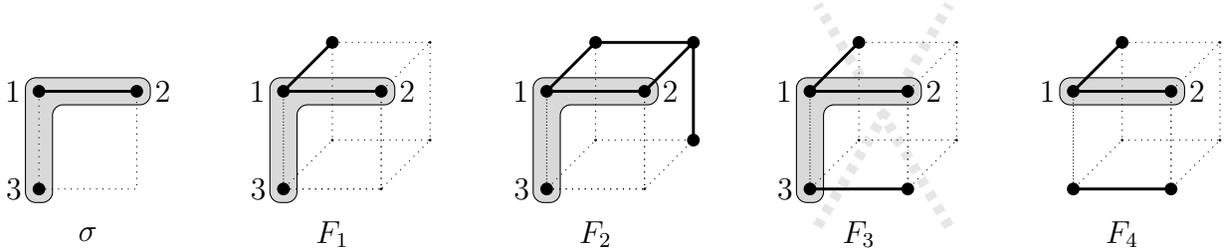

Let $\sigma$ be a type of dimension $r$. Let $G$ be a (large) 
$F$-free spanning subgraph of $\Q_n$, so $\dim(G) = n$. We say a type 
map $\theta$ to $G$ is a \emph{$\sigma$-type map} if there exists a feasible
bijection $f:Im(\theta) \rightarrow V(\sigma)$. Let $\Theta$ be the set of all $\sigma$-type maps $\theta$ to $G$.
Let $\F_k^{\sigma}$ be the set of all $F$-free $\sigma$-flags of dimension $k$.
Given a $\sigma$-flag $F_1=(H_1,\theta_1)\in\F^{\sigma}_k$ and a map $\theta\in \Theta$, 
we define $p(F_1,\theta;G)$ to be the probability that a feasible map $f:V(H_1)\rightarrow V(G)$ chosen uniformly at random subject to 
$f\cdot \theta_1 = \theta$
satisfies $(G[Im(f)],\theta)\simeq F_1$.
Note that if $(Im(\theta),\theta)\not\simeq \sigma$, then $p(F_1,\theta;G) = 0$. 
Given two $\sigma$-flags $F_1=(H_1,\theta_1)\in \F^{\sigma}_{k_1}$ and $F_2=(H_2,\theta_2)\in \F^{\sigma}_{k_2}$, 
for $\theta\in \Theta$, we define $p(F_1,F_2,\theta;G)$ to be the probability that 
if we choose two feasible maps $f_1:V(H_1)\rightarrow V(G)$ and $f_2:V(H_2)\rightarrow V(G)$ uniformly and independently at random  
subject to $f_1\cdot \theta_1 = \theta, f_2\cdot \theta_2 = \theta $ and $D(Im(f_1)) \cap D(Im(f_2)) = D(Im(\theta))$,
then $$(G[Im(f_1)],\theta) \simeq F_1 \text{ and } (G[Im(f_2)],\theta) \simeq F_2.$$ 

Note that $p(F_1,F_2,\theta;G)$ makes sense only when $n\ge k_1+k_2-r $ 
since $D\left(Im(f_1) \cup Im(f_2)\right) = D(Im(f_1)) \cup D(Im(f_2))$ must be a subset of $D(V(G))$.
When comparing $p(F_1,F_2,\theta;G)$ with $p(F_1,\theta;G)p(F_2,\theta;G)$, we see
that the only difference between these two probabilities is that in 
$p(F_1,\theta;G)p(F_2,\theta;G)$ we ask only for 
\begin{equation}
f_1\cdot \theta_1 = \theta \textrm{ and } f_2\cdot \theta_2 = \theta
 \label{eq:pp}
\end{equation}
where $f_1,f_2$ are two randomly chosen feasible maps, 
while in $p(F_1,F_2,\theta;G)$ 
we ask not only for \eqref{eq:pp} but also for
\begin{equation} 
D(Im(\theta)) = D(Im(f_1)) \cap D(Im(f_2)).
\label{eq:onlyp}
\end{equation} 
When $n$ is very large, intuitively, if \eqref{eq:pp} holds, 
then with high probability \eqref{eq:onlyp} also holds, and then the difference between
these two probabilities is negligible. This following lemma states it formally.
It is similar to Lemma 2.1 in~\cite{Baber:2011}, which is a special case of Lemma 2.3 in~\cite{Razborov:2007}.

\begin{lemma} \label{lem:productdensity}
For any $F_1=(H_1,\theta_1)\in \F^{\sigma}_{k_1}, F_2=(H_2,\theta_2)\in \F^{\sigma}_{k_2}$, $\theta\in \Theta$, and $G$ being a spanning subgraph of $\Q_n$ it holds that
\[
p(F_1,\theta;G)p(F_2,\theta;G) = p(F_1,F_2,\theta;G) + o(1)
\]
where the $o(1)$ term tends to $0$ as $n$ tends to infinity.
\end{lemma}


\begin{proof}
Choose two independent feasible maps $f_1:V(H_1)\rightarrow V(G)$ and $f_2:V(H_2)\rightarrow V(G)$ 
uniformly at random subject to $f_1\cdot \theta_1 = \theta$ and $f_2\cdot \theta_2 = \theta$. 
For such choices of $f_1$ and $f_2$, let $A$ be the event
\begin{equation*}
(G[Im(f_1)],\theta)\simeq F_1 \text{ and } (G[Im(f_2)],\theta) \simeq F_2 ,
\end{equation*}
 and $B$ be the event 
 \begin{equation*}
D(Im(f_1)) \cap D(Im(f_2)) = D(Im(\theta)) . 
\end{equation*}

We have $p(F_1,\theta; G)p(F_2,\theta; G)=P(A)$ and $p(F_1,F_2,\theta;G)=P(A|B)$. 
Using that for any $A$ and $B$, it holds that
$$P(A|B)P(B)=P(A\cap B)\le P(A)\le P(A\cap B)+P(\overline{B}), $$ 
we have $|P(A|B)P(B)-P(A)|\le P(\overline{B}).$ 
Hence it suffices to show $P(B)\ge 1-o(1)$. Note that $P(B)$ depends on $V(H_1), V(H_2), V(G)$
 but not on the edges of these graphs.
 
For $i=1,2$, let $\phi_i$ be the $\phi$ in Lemma~\ref{Dmap} for $f_i$.
We compute $P(B)$ by counting possible choices of $\phi_i$ instead of counting $f_i$'s directly.
We first consider the case that the type $\sigma\ne 0$, i.e., some vertex is labeled. 
From $f_i\cdot \theta_i = \theta$ we know that $\phi_i(D(Im(\theta_i)))=D(Im(\theta))$, so we next need to look at $\phi_i$ on $D(V(H_i))\setminus D(Im(\theta_i))$.
Recall that $d(Im(\theta)) = r$, 
hence there are still $k_i - r$ coordinates to be chosen from $[n] \setminus D(Im(\theta))$.

We know $f_i(\theta_i(1)) = \theta(1)$, so each $\phi_i$ gives one feasible 
map $f_i$. 
Note that different choices of $\phi_i$ may give the same $f_i$.
Let $M_i$ be the number of feasible maps $f_i': V(H_i)\rightarrow Q\left( V(H_i) \right)$ satisfying $f_i'\cdot\theta_i = \theta_i$.  
Observe that $M_i$ is also the number of $f_i$'s for each choice of $(k_i -r)$ coordinates from $[n] \setminus D(Im(\theta))$ 
given that $f_i\cdot \theta_i = \theta$. 
Note that good choices for the event $B$ are choosing coordinates for $\phi_1(D(V(H_1))\setminus D(Im(\theta_1)))$ and $\phi_2(D(V(H_2))\setminus D(Im(\theta_2)))$ that are disjoint.
So we can compute that
$$P(B) = \frac{{n-r\choose k_1-r}M_1{n-k_1\choose k_2-r}M_2}{{n-r\choose k_1-r}M_1{n-r\choose k_2-r}M_2}=1-o(1).$$

For the case $\sigma = 0$, each choice of $\phi_i$ will give $2^{n}$ different $f_i$'s, so we have 
$$P(B) = \frac{{n\choose k_1}M_12^{n}{n-k_1\choose k_2}M_22^{n}}{{n\choose k_1}M_12^{n}{n\choose k_2}M_22^{n}}=1-o(1).$$

\end{proof}

Now we can use this version of the flag algebra method to compute $\ex_{\Q}(F)$. This is the same as in \cite{Baber:2011}. 
We suggest the reader to start reading the next section in parallel with the following 
text as the entire next section can be viewed as an example.

Fix a type $\sigma \ne 0$. Averaging over a uniformly and randomly chosen $\theta\in\Theta$ we have
\begin{equation}
\E_{\theta\in \Theta}[p(F_1,\theta;G)p(F_2,\theta;G)] = \E_{\theta\in \Theta}[p(F_1,F_2,\theta;G)] + o(1).
\label{eq:prob_estimate}
\end{equation}

Pick $s\ge k_1+k_2-r$. For $H\in \mathcal{H}_s$, let $\Theta_H$ be the set of all $\sigma$-type maps to $H$. Then

\begin{equation}
 \E_{\theta \in \Theta}[p(F_1,F_2,\theta;G)]= \sum_{H \in \mathcal{H}_s}\E_{\theta \in \Theta_H}[p(F_1,F_2,\theta;H)]p(H,G).
 \label{eq:prob_average}
\end{equation}

We pick $\sigma \ne 0$ simply because if $\sigma = 0$, then \eqref{eq:prob_average} does not hold.
Let $\F =\left\{ F_1,\dots,F_{\ell} \right\}\subseteq \F^{\sigma}_k$ be satisfying 
\begin{equation}
s\ge 2k - r
 \label{eq:s}
\end{equation}
and let $M=(m_{ij})$ be a positive semidefinite $\ell$-by-$\ell$ matrix. 
For $\theta\in \Theta$ define $\mathbf{p}_{\theta}=\left\{ p(F_1,\theta;G),\dots,p(F_{\ell},\theta;G) \right\}$. 
Using \eqref{eq:prob_estimate} and \eqref{eq:prob_average}, we have 
\begin{equation}
0 \le \E_{\theta\in\Theta}[\mathbf{p}_{\theta}M\mathbf{p}^T_{\theta}] = 
\sum_{1\le i,j\le \ell} \sum_{H\in \mathcal{H}_s} m_{ij} \E_{\theta\in\Theta_H}[p(F_i,F_j,\theta;H)]p(H,G) + o(1).
 \label{eq:positive}
\end{equation}
For $H\in \mathcal{H}_s$ we define $c_H(\sigma,\F,M)$ to be the coefficient of $p(H,G)$ in $\eqref{eq:positive}$ i.e.,
\begin{equation*}
c_H(\sigma,\F,M) = \sum_{1\le i,j\le \ell} m_{ij} \E_{\theta\in\Theta_H}[p(F_i,F_j,\theta;H)].
 \label{eq:cH}
\end{equation*}

Then we can rewrite $\eqref{eq:positive}$ as
\begin{equation*}
0 \le \sum_{H\in \mathcal{H}_s} c_H(\sigma,\F,M) p(H,G) + o(1).
 \label{eq:positive_coeff}
\end{equation*}

Fix $G$ and $\mathcal{H}_s$, suppose we have $t$ choices of $(\sigma_i,\F_i,M_i)$, 
where each $\sigma_i\ne 0$ is a type of dimension $r_i$,
each $\F_i$ is a subset of $\F_{k_i}^{\sigma_i}$ satisfying $s\ge 2k_i -r_i$,
and each $M_i$ is a positive semidefinite matrix of dimension $|\F_i|$.
 Then for $H\in\mathcal{H}_s$ we have 
\[
0 \le \sum_{H\in \mathcal{H}_s}\left( \sum_{i=1}^t c_H(\sigma_i,\F_i,M_i) \right)p(H,G) + o(1).
\]

Define $c_H = \sum_{i=1}^t c_H(\sigma_i, \F_i
,M_i) $, then we have $0 \le \sum_{H\in \mathcal{H}_s}c_H p(H,G) + o(1)$. Together with \eqref{eq:density}, we have
\[
\d(G)\le \sum_{H\in \mathcal{H}_s}(\d(H)+c_H )p(H,G) + o(1).
\]

Thus $\d(G)\le \max_{H\in \mathcal{H}_s}(\d(H)+c_H )+o(1)$ and therefore $\pi_{\Q}(F) \le \max_{H\in \mathcal{H}_s}(\d(H)+c_H )$.

\section{\texorpdfstring{Example for $\Q_2$ }{Example for Q2}}\label{sec:example}
 

In this section we apply the flag algebra method with $F=C_4$ and $\cH_2$.
We obtain a weaker bound $\pi_{\Q}(C_4) \leq 2/3$ than in Theorem~\ref{thm:main}.
On the other hand, it allows us to present the proof with all the details and
hopefully it makes the reader more comfortable while reading the proofs of Theorems~\ref{thm:main} and \ref{thm:C6}
as the method is the same.
 

We consider only one type, a single labelled vertex, so its dimension is zero.
As flags $\F=\{F_0, F_1\}$ we use both possible flags on two vertices with
one labelled vertex and containing $0$ and $1$ edges, respectively. 
So they both have dimension one.
See Figure~\ref{fig-F} for $F_0$ and $F_1$.

\begin{figure}[ht]
\begin{center}
\def\e{1}
\def\lab{-1}
\def\lw{1.1pt}

\def\H
{
\draw[fill=gray!30]
(0,0)  circle (6pt)
;
\draw
(0,0)  node[insep] (a) {} 
 ++(\e,0) node[insep] (b) {}
(a) ++(0,-0.5) node {1}
;

}

\def\confA{
\H
;
\draw
(0.5*\e,\lab) node {$F_0$}
;
\draw[dotted]
(a)--(b);
}

\def\confB{
\H
\draw[line width=\lw]
(a) -- (b)
;
\draw
(0.5*\e,\lab) node {$F_1$}
;
}

\begin{tikzpicture}
[insep/.style={inner sep=1.7pt, outer sep=0pt, circle, fill}, 
noin/.style={inner sep=0pt, outer sep=0pt, circle, fill},
list/.style={inner sep=0.0pt, outer sep=0pt, rectangle,fill=white},
precol/.style={inner sep=1.8pt, outer sep=0pt, circle, fill},
three/.style={inner sep=1.5pt, outer sep=0pt, regular polygon,regular polygon sides=3, draw,fill=white},
four/.style={inner sep=1.8pt, outer sep=0pt, regular polygon,regular polygon sides=4, draw,fill=white},
five/.style={inner sep=1.8pt, outer sep=0pt, regular polygon,regular polygon sides=5, draw,fill=white}
]

\begin{scope}[xshift=0cm, yshift = 0cm] 
\confA 
\end{scope}
\begin{scope}[xshift=3cm, yshift = 0cm] 
\confB
\end{scope}

\end{tikzpicture}
\end{center}
\caption{Two flags of dimension one with one labeled vertex.}
\label{fig-F}
\end{figure}
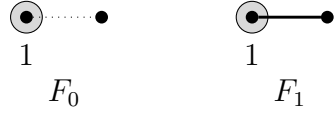

Recall that $\mathcal{H}_2$ is the set of all $C_4$-free subgraphs of $\Q_2$. 
See Figure~\ref{fig-H} for the list of all five of them.
Note that the variables corresponding to the previous section are $r =  0, k = 1, s = 2$ and $t = 1$.
We can use $\mathcal{H}_2$ because \eqref{eq:s} holds.

\begin{figure}[ht]
\begin{center}
\def\e{1}
\def\lab{-1}
\def\lw{1.1pt}

\def\H
{
\draw[dotted]
(0,0)  node[insep] (a) {} --
 ++(\e,0) node[insep] (b) {} --
 ++(0,\e) node[insep] (c) {} --
 ++(-\e,0) node[insep] (d) {} -- (a)
;
}

\def\confA{
\H
;
\draw
(0.5*\e,\lab) node {$H_0$}
;
}

\def\confB{
\H
\draw[line width=\lw]
(a) -- (b)
;
\draw
(0.5*\e,\lab) node {$H_1$}
;
}

\def\confC{
\H
\draw[line width=\lw]
(a) -- (b) -- (c)
;
\draw
(0.5*\e,\lab) node {$H_3$}
;
}
\def\confD{
\H
\draw[line width=\lw]
(a) -- (b) (c) -- (d)
;
\draw
(0.5*\e,\lab) node {$H_{2}$}
;
}
\def\confE{
\H
\draw[line width=\lw]
(a) -- (b) -- (c) -- (d) 
;
\draw
(0.5*\e,\lab) node {$H_4$}
;
}

\begin{tikzpicture}
[insep/.style={inner sep=1.7pt, outer sep=0pt, circle, fill}, 
noin/.style={inner sep=0pt, outer sep=0pt, circle, fill},
list/.style={inner sep=0.0pt, outer sep=0pt, rectangle,fill=white},
precol/.style={inner sep=1.8pt, outer sep=0pt, circle, fill},
three/.style={inner sep=1.5pt, outer sep=0pt, regular polygon,regular polygon sides=3, draw,fill=white},
four/.style={inner sep=1.8pt, outer sep=0pt, regular polygon,regular polygon sides=4, draw,fill=white},
five/.style={inner sep=1.8pt, outer sep=0pt, regular polygon,regular polygon sides=5, draw,fill=white}
]

\begin{scope}[xshift=0cm, yshift = 0cm] 
\confA 
\end{scope}
\begin{scope}[xshift=2.7cm, yshift = 0cm] 
\confB
\end{scope}
\begin{scope}[xshift=5.4cm, yshift = 0cm] 
\confD
\end{scope}
\begin{scope}[xshift=8.1cm, yshift = 0cm] 
\confC
\end{scope}
\begin{scope}[xshift=10.8cm, yshift = 0cm] 
\confE
\end{scope}

\end{tikzpicture}
\end{center}
\caption{$C_4$-free spanning subgraphs of $\Q_2$.}
\label{fig-H}
\end{figure}
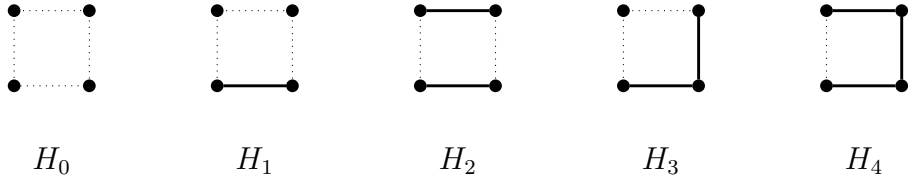

In order to calculate the coefficients $c_H$ we need to compute
$\E_{\theta\in \Theta}p(F_i,F_j,\theta,H)$ for all possible
$H \in \cH_2$ and $F_i,F_j \in \F$.
The values of $\E_{\theta\in \Theta}p(F_i,F_j,\theta,H)$ are given in Table~\ref{tab-E}.

\begin{table}[ht]
\begin{center}
    \begin{tabular}{ | c | c | c | c | c | c |}
    \hline
     &$H_0$&$H_1$&$H_2$&$H_3$&$H_4$\\ \hline
    $F_0,F_0$&1&1/2&0&1/4&0\\ \hline
    $F_0,F_1$&0&1/4&1/2&1/4&1/4\\ \hline
    $F_1,F_1$&0&0&0&1/4&1/2\\ \hline
    \end{tabular}
\end{center}
\caption{ $\E_{\theta\in \Theta}p(F_i,F_j,\theta,H)$.}
\label{tab-E}
\end{table}

We show how to compute  $\E_{\theta\in \Theta}p(F_0,F_1,\theta,H_3)$ and leave
the verification of other entries in Table~\ref{tab-E} to the interested readers. 
In this case we need
to compute the probability that a uniformly and randomly chosen $\theta \in \Theta$
and two pairs of vertices with Hamming distance one $V_0,V_1 \subset V(H_3)$
chosen independently and uniformly at random with intersection $Im(\theta)$ induce flags $(H_3[V_0],\theta)$ and
$(H_3[V_1],\theta)$ that are isomorphic to $F_0$ and $F_1$, respectively.
By inspection of the cases, this happens only when $Im(\theta)$ is a vertex
of degree one and the other vertices of $V_0$ are $V_1$ are the vertices of degree
zero and two, respectively. So 2 out of 8 possibilities are satisfying the condition.

As $l = 2$, we want to choose a positive semidefinite $2\times2$ matrix $M$ used 
in \eqref{eq:positive}. In the general form
$$M =
\left( \begin{array}{cc}
m_{11} & m_{12} \\m_{21} & m_{22}
\end{array}\right ).$$
Note that $m_{12}  = m_{21}$ as $M$ must be symmetric. 
We can compute $c_H(\sigma,\F,M)$ by multiplying the vector $(m_{11}, 2m_{12}, m_{22})$ 
with the column corresponding to $H$ in Table~\ref{tab-E} for every $H \in \cH_2$. Note that $c_H(\sigma,\F,M)$ is the same as $c_H$ because $t=1$.
Together with densities we have
\begin{eqnarray*}
\d(H_0) + c_{H_0} &=& 0 + m_{11}\\
\d(H_1) + c_{H_1} &=& 1/4 + m_{11}/2 + m_{12}/2\\
\d(H_2) + c_{H_2} &=& 1/2 + m_{12}\\
\d(H_3) + c_{H_3} &=& 1/2 + m_{11}/4 + m_{12}/2 + m_{22}/4\\
\d(H_4) + c_{H_4} &=& 3/4 + m_{12}/2 + m_{22}/2.\\
\end{eqnarray*}
Recall that $\pi_{\Q}(C_4) \le \max_i (\d(H_i) + c_{H_i})$. So we want to minimize $\max_i (\d(H_i) + c_{H_i})$ over all positive semidefinite matrices.   
This can be expressed as a semidefinite program $(P)$ as follows:

$$
(P) 
\begin{cases}
\text{Minimize}\,\,\, v\,\\
\text{subject to}\,\,\,  v \geq \d(H_i) + c_{H_i}\,\,  \forall H_i\in \cH_2 \\
v \in \R, M\text{ is positive semidefinite.}
\end{cases}
$$


The optimal solution of $(P)$ is 
$$M^*=
\left( \begin{array}{cc}
2/3 & -1/3 \\ -1/3 & 1/6
\end{array}\right )$$
and it gives $\max_i (\d(H_i) + c_{H_i}) = 2/3$. Note that it is not necessary to use the optimal solution to get
an upper bound but any feasible solution gives an upper bound (of course, not as good the optimal solution). 
We use this observation later in order to fix rounding errors by CSDP solver.

\section{\texorpdfstring{Proof of Theorem~\ref{thm:main}}{Proof of Theorem 1}}\label{sec:proof}

The proof of Theorem~\ref{thm:main} goes along the same lines
as the proof in the previous section. It is just performed with
$\Q_3$ and with more flags.

Let $E_0,E_1 \subseteq \Q_1$ be cube graphs
with zero and one edge, respectively and let $\theta_i: [2] \rightarrow V(E_i)$
for $i \in \{0,1\}$. We consider two types $\sigma_0 = (E_0,\theta_0)$ and 
$\sigma_1 = (E_1,\theta_1)$ and flags of dimension two. 
Let $\F_0= \{F_0^0,\ldots,F_7^0\}$ be all flags in $\F^{\sigma_0}_2$ on $4$ vertices
and let $\F_1= \{F_0^1,\ldots,F_6^1\}$ be all flags in $\F^{\sigma_1}_2$ on $4$ vertices.
The flag of type $\sigma_1$ with four edges is not in $\F^{\sigma_1}_2$
since it is not $C_4$-free.
See Figure~\ref{fig-FQ2} for the list of flags.

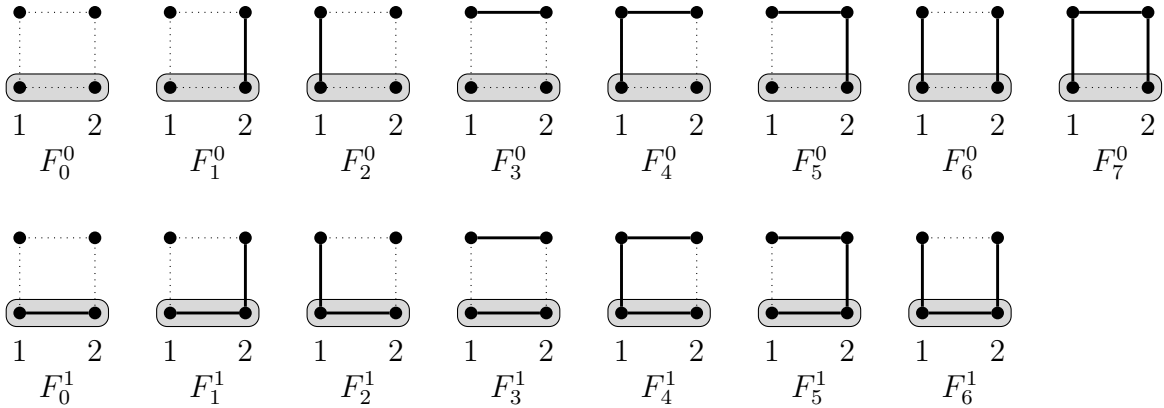
\begin{figure}[ht]
\begin{center}
\def\e{1}
\def\lab{-1}
\def\lw{1.1pt}
\def\msh{0.18}

\def\H
{
\draw[fill=gray!30, ,rounded corners=4pt]
(-\msh,-\msh) rectangle (\e+\msh, \msh)
;
\draw[dotted]
(0,0)  node[insep] (a) {} 
-- ++(\e,0) node[insep] (b) {}
-- ++(0,\e) node[insep] (c) {}
--  ++(-\e,0) node[insep] (d) {}
--(a)
(a) ++(0,-0.5) node {1}
(b) ++(0,-0.5) node {2}
;
}

\def\confA{
\H
\draw[line width=\lw] (a) -- (b) ;
\draw (0.5*\e,\lab) node {$F_0^1$};
}

\def\confB{
\H
\draw[line width=\lw] (a) -- (b)--(c) ;
\draw (0.5*\e,\lab) node {$F_1^1$};
}

\def\confC{
\H
\draw[line width=\lw] (d)--(a) -- (b) ;
\draw (0.5*\e,\lab) node {$F_2^1$};
}

\def\confD{
\H
\draw[line width=\lw] (c)--(d)(a) -- (b) ;
\draw (0.5*\e,\lab) node {$F_3^1$};
}

\def\confE{
\H
\draw[line width=\lw] (c)--(d)--(a) -- (b) ;
\draw (0.5*\e,\lab) node {$F_4^1$};
}

\def\confF{
\H
\draw[line width=\lw] (a) -- (b)--(c)--(d);
\draw (0.5*\e,\lab) node {$F_5^1$};
}

\def\confG{
\H
\draw[line width=\lw] (d)--(a) -- (b)-- (c);
\draw (0.5*\e,\lab) node {$F_6^1$};
}

\def\confH{
\H
\draw[line width=\lw] (d)--(a)-- (b)-- (c)--(d);
\draw (0.5*\e,\lab) node {$F_7^1$};
}

\def\confAA{
\H
\draw[line width=\lw] (a)  (b) ;
\draw (0.5*\e,\lab) node {$F_0^0$};
}

\def\confBB{
\H
\draw[line width=\lw] (a)  (b)--(c) ;
\draw (0.5*\e,\lab) node {$F_1^0$};
}

\def\confCC{
\H
\draw[line width=\lw] (d)--(a)  (b) ;
\draw (0.5*\e,\lab) node {$F_2^0$};
}

\def\confDD{
\H
\draw[line width=\lw] (c)--(d)(a)  (b) ;
\draw (0.5*\e,\lab) node {$F_3^0$};
}

\def\confEE{
\H
\draw[line width=\lw] (c)--(d)--(a)  (b) ;
\draw (0.5*\e,\lab) node {$F_4^0$};
}

\def\confFF{
\H
\draw[line width=\lw] (a)  (b)--(c)--(d);
\draw (0.5*\e,\lab) node {$F_5^0$};
}

\def\confGG{
\H
\draw[line width=\lw] (d)--(a)  (b)-- (c);
\draw (0.5*\e,\lab) node {$F_6^0$};
}

\def\confHH{
\H
\draw[line width=\lw] (d)--(a)  (b)-- (c)--(d);
\draw (0.5*\e,\lab) node {$F_7^0$};
}

\begin{tikzpicture}
[insep/.style={inner sep=1.7pt, outer sep=0pt, circle, fill}, 
noin/.style={inner sep=0pt, outer sep=0pt, circle, fill},
list/.style={inner sep=0.0pt, outer sep=0pt, rectangle,fill=white},
precol/.style={inner sep=1.8pt, outer sep=0pt, circle, fill},
three/.style={inner sep=1.5pt, outer sep=0pt, regular polygon,regular polygon sides=3, draw,fill=white},
four/.style={inner sep=1.8pt, outer sep=0pt, regular polygon,regular polygon sides=4, draw,fill=white},
five/.style={inner sep=1.8pt, outer sep=0pt, regular polygon,regular polygon sides=5, draw,fill=white}
]

\begin{scope}[xshift=0cm, yshift = 0cm] 
\confA 
\end{scope}
\begin{scope}[xshift=2cm, yshift = 0cm] 
\confB
\end{scope}
\begin{scope}[xshift=4cm, yshift = 0cm] 
\confC
\end{scope}
\begin{scope}[xshift=6cm, yshift = 0cm] 
\confD
\end{scope}
\begin{scope}[xshift=8cm, yshift = 0cm] 
\confE
\end{scope}
\begin{scope}[xshift=10cm, yshift = 0cm] 
\confF
\end{scope}
\begin{scope}[xshift=12cm, yshift = 0cm] 
\confG
\end{scope}

\def\sh{3cm}

\begin{scope}[xshift=0cm, yshift = \sh] 
\confAA 
\end{scope}
\begin{scope}[xshift=2cm, yshift = \sh] 
\confBB
\end{scope}
\begin{scope}[xshift=4cm, yshift = \sh] 
\confCC
\end{scope}
\begin{scope}[xshift=6cm, yshift = \sh] 
\confDD
\end{scope}
\begin{scope}[xshift=8cm, yshift = \sh] 
\confEE
\end{scope}
\begin{scope}[xshift=10cm, yshift =\sh] 
\confFF
\end{scope}
\begin{scope}[xshift=12cm, yshift = \sh] 
\confGG
\end{scope}
\begin{scope}[xshift=14cm, yshift = \sh] 
\confHH
\end{scope}
\end{tikzpicture}
\end{center}
\caption{$\F_0$ is in the first row and $\F_1$ is in the second row. }
\label{fig-FQ2}
\end{figure}

Next we need to obtain $\cH_3$, the set of all $C_4$-free subgraphs
of $\Q_3$. We wrote two independent computer programs for generating 
the graphs and obtained a list of $99$ graphs which agrees with \cite{Thomason:2009} 
where the authors also obtained $99$ such graphs.

Our computer programs also calculated $\E_{\theta\in \Theta}p(F_i^k,F_j^k,\theta,H)$ for all possible
$H \in \cH_3$ and $F_i^k,F_j^k \in \F_k$
and produced a semidefinite program.

The resulting semidefinite program was solved by CSDP~\cite{Borchers:1999}.
Due to rounding, the resulting matrix $M^*$ may not be positive semidefinite. 
We used MATLAB to perturb the matrix to make sure that it is positive semidefinite and then 
we computed an upper bound $\pi_{\Q}(C_4) \leq 0.6068$.

\section{\texorpdfstring{Proof of Theorem~\ref{thm:C6}}{Proof of Theorem 2}}\label{sec:proofC6}

The proof of Theorem~\ref{thm:C6} is the same as the proof of Theorem~\ref{thm:main}.
We also considered both types of dimension one with two labeled vertices.
In this case we again considered all possible flags on four vertices.
See Figure~\ref{fig-FQ2C6} for the list of the flags.

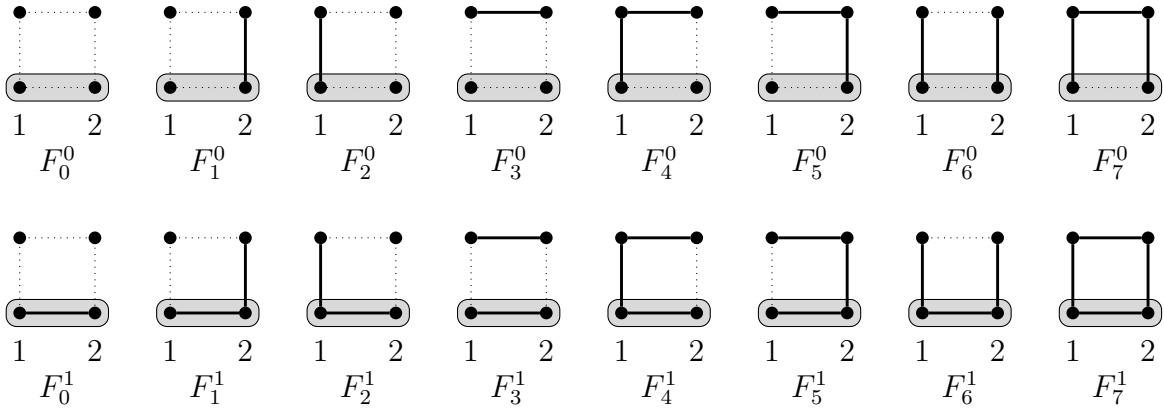
\begin{figure}[ht]
\begin{center}
\def\e{1}
\def\lab{-1}
\def\lw{1.1pt}
\def\msh{0.18}

\def\H
{
\draw[fill=gray!30, ,rounded corners=4pt]
(-\msh,-\msh) rectangle (\e+\msh, \msh)
;
\draw[dotted]
(0,0)  node[insep] (a) {} 
-- ++(\e,0) node[insep] (b) {}
-- ++(0,\e) node[insep] (c) {}
--  ++(-\e,0) node[insep] (d) {}
--(a)
(a) ++(0,-0.5) node {1}
(b) ++(0,-0.5) node {2}
;
}

\def\confA{
\H
\draw[line width=\lw] (a) -- (b) ;
\draw (0.5*\e,\lab) node {$F_0^1$};
}

\def\confB{
\H
\draw[line width=\lw] (a) -- (b)--(c) ;
\draw (0.5*\e,\lab) node {$F_1^1$};
}

\def\confC{
\H
\draw[line width=\lw] (d)--(a) -- (b) ;
\draw (0.5*\e,\lab) node {$F_2^1$};
}

\def\confD{
\H
\draw[line width=\lw] (c)--(d)(a) -- (b) ;
\draw (0.5*\e,\lab) node {$F_3^1$};
}

\def\confE{
\H
\draw[line width=\lw] (c)--(d)--(a) -- (b) ;
\draw (0.5*\e,\lab) node {$F_4^1$};
}

\def\confF{
\H
\draw[line width=\lw] (a) -- (b)--(c)--(d);
\draw (0.5*\e,\lab) node {$F_5^1$};
}

\def\confG{
\H
\draw[line width=\lw] (d)--(a) -- (b)-- (c);
\draw (0.5*\e,\lab) node {$F_6^1$};
}

\def\confH{
\H
\draw[line width=\lw] (d)--(a)-- (b)-- (c)--(d);
\draw (0.5*\e,\lab) node {$F_7^1$};
}

\def\confAA{
\H
\draw[line width=\lw] (a)  (b) ;
\draw (0.5*\e,\lab) node {$F_0^0$};
}

\def\confBB{
\H
\draw[line width=\lw] (a)  (b)--(c) ;
\draw (0.5*\e,\lab) node {$F_1^0$};
}

\def\confCC{
\H
\draw[line width=\lw] (d)--(a)  (b) ;
\draw (0.5*\e,\lab) node {$F_2^0$};
}

\def\confDD{
\H
\draw[line width=\lw] (c)--(d)(a)  (b) ;
\draw (0.5*\e,\lab) node {$F_3^0$};
}

\def\confEE{
\H
\draw[line width=\lw] (c)--(d)--(a)  (b) ;
\draw (0.5*\e,\lab) node {$F_4^0$};
}

\def\confFF{
\H
\draw[line width=\lw] (a)  (b)--(c)--(d);
\draw (0.5*\e,\lab) node {$F_5^0$};
}

\def\confGG{
\H
\draw[line width=\lw] (d)--(a)  (b)-- (c);
\draw (0.5*\e,\lab) node {$F_6^0$};
}

\def\confHH{
\H
\draw[line width=\lw] (d)--(a)  (b)-- (c)--(d);
\draw (0.5*\e,\lab) node {$F_7^0$};
}

\begin{tikzpicture}
[insep/.style={inner sep=1.7pt, outer sep=0pt, circle, fill}, 
noin/.style={inner sep=0pt, outer sep=0pt, circle, fill},
list/.style={inner sep=0.0pt, outer sep=0pt, rectangle,fill=white},
precol/.style={inner sep=1.8pt, outer sep=0pt, circle, fill},
three/.style={inner sep=1.5pt, outer sep=0pt, regular polygon,regular polygon sides=3, draw,fill=white},
four/.style={inner sep=1.8pt, outer sep=0pt, regular polygon,regular polygon sides=4, draw,fill=white},
five/.style={inner sep=1.8pt, outer sep=0pt, regular polygon,regular polygon sides=5, draw,fill=white}
]

\begin{scope}[xshift=0cm, yshift = 0cm] 
\confA 
\end{scope}
\begin{scope}[xshift=2cm, yshift = 0cm] 
\confB
\end{scope}
\begin{scope}[xshift=4cm, yshift = 0cm] 
\confC
\end{scope}
\begin{scope}[xshift=6cm, yshift = 0cm] 
\confD
\end{scope}
\begin{scope}[xshift=8cm, yshift = 0cm] 
\confE
\end{scope}
\begin{scope}[xshift=10cm, yshift = 0cm] 
\confF
\end{scope}
\begin{scope}[xshift=12cm, yshift = 0cm] 
\confG
\end{scope}
\begin{scope}[xshift=14cm, yshift = 0cm] 
\confH
\end{scope}

\def\sh{3cm}

\begin{scope}[xshift=0cm, yshift = \sh] 
\confAA 
\end{scope}
\begin{scope}[xshift=2cm, yshift = \sh] 
\confBB
\end{scope}
\begin{scope}[xshift=4cm, yshift = \sh] 
\confCC
\end{scope}
\begin{scope}[xshift=6cm, yshift = \sh] 
\confDD
\end{scope}
\begin{scope}[xshift=8cm, yshift = \sh] 
\confEE
\end{scope}
\begin{scope}[xshift=10cm, yshift =\sh] 
\confFF
\end{scope}
\begin{scope}[xshift=12cm, yshift = \sh] 
\confGG
\end{scope}
\begin{scope}[xshift=14cm, yshift = \sh] 
\confHH
\end{scope}
\end{tikzpicture}
\end{center}
\caption{Flags used in the proof of Theorem~\ref{thm:C6}.}
\label{fig-FQ2C6}
\end{figure}

Next we need to obtain $\cH_3$, the set of all $C_6$-free subgraphs
of $\Q_3$. We wrote two independent computer programs for generating 
the graphs and obtained a list of 116 graphs. We again used CSDP
solver and after perturbation we obtained that 
$\pi_{\Q}(C_6) \leq 0.3755 $.

\section{Middle layers}


This section describes the idea of proving Theorem~\ref{thm:middle}. We do not
give the entire proof as it is computer assisted. Instead, we show a proof of
a weaker result which goes along the same way as the proof of Theorem~\ref{thm:middle}.
Note that it is easy to see that it is sufficient to show the theorem only for the middle three layers and we are giving
an upper bound.

We start with describing the upper bound $(3+\sqrt{2})N/2$ using flag algebras.
We skip some technical details; namely stating and proving a lemma analogous to Lemma~\ref{lem:productdensity} for hypercubes.

Let $A_n, B_n, C_n$ be the family of subsets of $[n]$ having sizes $\lfloor n/2\rfloor-1, \lfloor n/2\rfloor$ and $\lfloor n/2\rfloor+1$ respectively. Let $M_n = A_n\cup B_n\cup C_n$, then $|M_n| = (3+o(1))N$. Given a subset $G_n$ of $M_n$, define
\[\d(G_n) = \frac{|G_n\cap A_n|}{|A_n|} + \frac{|G_n\cap B_n|}{|B_n|} + \frac{|G_n\cap C_n|}{|C_n|}.\]
In the following we view a family of subsets as its Hasse diagram. This allows us
to talk about subsets as vertices and edges for subsets that differ by exactly one element.
Let $\cH_n$ be the family of all $Q_2$-free subsets of $M_n$, then we can write the result in \cite{Axenovich:2011} as
$$\lim_{n\to\infty,G_n\in \cH_n} \d(G_n)\le (3+\sqrt{2})/2.$$
The same result can be achieved by considering $\cH_2$ (see Figure~\ref{fig-middH}), and two flags (see~Figure~\ref{fig-middF}). 
\begin{figure}
\begin{center}
{\def\e{0.6}
\def\c{ circle(4pt) }
\def\s{1.9}
\def\lab{-.6}

\def\confA{
\draw (0,0)   -- (\e,\e) -- (0,2*\e) -- (-\e,\e) -- (0,0) ;
\fill (0,0)  \c   (\e,\e) \c   (0,2*\e)  \c   (-\e,\e) \c ;
\filldraw[fill=white] (0,0)  \c   (\e,\e)   \c (0,2*\e)   \c  (-\e,\e)  \c;
\draw (0,\lab) node {$H_0$} ;
}

\def\confB{
\draw (0,0)   -- (\e,\e) -- (0,2*\e) -- (-\e,\e) -- (0,0) ;
\fill (0,0)  \c   (\e,\e) \c   (0,2*\e)  \c   (-\e,\e) \c ;
\filldraw[fill=white] (0,0)   \c  (\e,\e)  \c  (0,2*\e)     (-\e,\e)  \c;
\draw (0,\lab) node {$H_1$} ;
}

\def\confBB{
\draw (0,0)   -- (\e,\e) -- (0,2*\e) -- (-\e,\e) -- (0,0) ;
\fill (0,0)  \c   (\e,\e) \c   (0,2*\e)  \c   (-\e,\e) \c ;
\filldraw[fill=white] (0,0)     (\e,\e)  \c  (0,2*\e)  \c    (-\e,\e)  \c;
\draw (0,\lab) node {$H_2$} ;
}

\def\confC{
\draw (0,0)   -- (\e,\e) -- (0,2*\e) -- (-\e,\e) -- (0,0) ;
\fill (0,0)  \c   (\e,\e) \c   (0,2*\e)  \c   (-\e,\e) \c ;
\filldraw[fill=white] (0,0)   \c  (\e,\e)    \c (0,2*\e)   \c  (-\e,\e)  ;
\draw (0,\lab) node {$H_3$} ;

}
\def\confD{
\draw (0,0)   -- (\e,\e) -- (0,2*\e) -- (-\e,\e) -- (0,0) ;
\fill (0,0)  \c   (\e,\e) \c   (0,2*\e)  \c   (-\e,\e) \c ;
\filldraw[fill=white] (0,0)   \c  (\e,\e)    (0,2*\e)   \c  (-\e,\e)  ;
\draw (0,\lab) node {$H_4$} ;
}

\def\confE{
\draw (0,0)   -- (\e,\e) -- (0,2*\e) -- (-\e,\e) -- (0,0) ;
\fill (0,0)  \c   (\e,\e) \c   (0,2*\e)  \c   (-\e,\e) \c ;
\filldraw[fill=white] (0,0) \c    (\e,\e)  \c  (0,2*\e)     (-\e,\e)  ;
\draw (0,\lab) node {$H_5$} ;
}

\def\confEE{
\draw (0,0)   -- (\e,\e) -- (0,2*\e) -- (-\e,\e) -- (0,0) ;
\fill (0,0)  \c   (\e,\e) \c   (0,2*\e)  \c   (-\e,\e) \c ;
\filldraw[fill=white] (0,0)   (\e,\e)  \c  (0,2*\e)  \c   (-\e,\e)  ;
\draw (0,\lab) node {$H_6$} ;
}

\def\confF{
\draw (0,0)   -- (\e,\e) -- (0,2*\e) -- (-\e,\e) -- (0,0) ;
\fill (0,0)  \c   (\e,\e) \c   (0,2*\e)  \c   (-\e,\e) \c ;
\filldraw[fill=white] (0,0)     (\e,\e) \c   (0,2*\e)     (-\e,\e) \c ;
\draw (0,\lab) node {$H_7$} ;
}
\def\confG{
\draw (0,0)   -- (\e,\e) -- (0,2*\e) -- (-\e,\e) -- (0,0) ;
\fill (0,0)  \c   (\e,\e) \c   (0,2*\e)  \c   (-\e,\e) \c ;
\filldraw[fill=white] (0,0)     (\e,\e)  \c  (0,2*\e)     (-\e,\e)  ;
\draw (0,\lab) node {$H_8$} ;
}
\def\confH{
\draw (0,0)   -- (\e,\e) -- (0,2*\e) -- (-\e,\e) -- (0,0) ;
\fill (0,0)  \c   (\e,\e) \c   (0,2*\e)  \c   (-\e,\e) \c ;
\filldraw[fill=white] (0,0)  \c   (\e,\e)    (0,2*\e)     (-\e,\e)  ;
\draw (0,\lab) node {$H_9$} ;
}
\def\confHH{
\draw (0,0)   -- (\e,\e) -- (0,2*\e) -- (-\e,\e) -- (0,0) ;
\fill (0,0)  \c   (\e,\e) \c   (0,2*\e)  \c   (-\e,\e) \c ;
\filldraw[fill=white] (0,0)    (\e,\e)    (0,2*\e)   \c    (-\e,\e)  ;
\draw (0,\lab) node {$H_{10}$} ;
}
\begin{tikzpicture}
[insep/.style={inner sep=1.7pt, outer sep=0pt, circle, fill}, 
noin/.style={inner sep=0pt, outer sep=0pt, circle, fill},
list/.style={inner sep=0.0pt, outer sep=0pt, rectangle,fill=white},
precol/.style={inner sep=1.8pt, outer sep=0pt, circle, fill},
three/.style={inner sep=1.5pt, outer sep=0pt, regular polygon,regular polygon sides=3, draw,fill=white},
four/.style={inner sep=1.8pt, outer sep=0pt, regular polygon,regular polygon sides=4, draw,fill=white},
five/.style={inner sep=1.8pt, outer sep=0pt, regular polygon,regular polygon sides=5, draw,fill=white}
]

\begin{scope}[xshift=0cm, yshift = 0cm] 
\confA 
\end{scope}
\begin{scope}[xshift=\s cm, yshift = 0cm] 
\confB
\end{scope}
\begin{scope}[xshift=2*\s cm, yshift = 0cm] 
\confBB
\end{scope}
\begin{scope}[xshift=3*\s cm, yshift = 0cm] 
\confC
\end{scope}
\begin{scope}[xshift=4*\s cm, yshift = 0cm] 
\confD
\end{scope}
\begin{scope}[xshift=5*\s cm, yshift = 0cm] 
\confE
\end{scope}
\begin{scope}[xshift=0.5*\s cm, yshift = -2.8cm] 
\confEE
\end{scope}
\begin{scope}[xshift=1.5*\s cm, yshift = -2.8cm] 
\confF
\end{scope}
\begin{scope}[xshift=2.5*\s cm, yshift =-2.8cm] 
\confG
\end{scope}
\begin{scope}[xshift=3.5*\s cm, yshift = -2.8cm] 
\confH
\end{scope}
\begin{scope}[xshift=4.5*\s cm, yshift =-2.8cm] 
\confHH
\end{scope}
\end{tikzpicture}}
\end{center}
\caption{$\cH_2$: $Q_2$-free subsets of $M_2$.}
\label{fig-middH}
\end{figure}
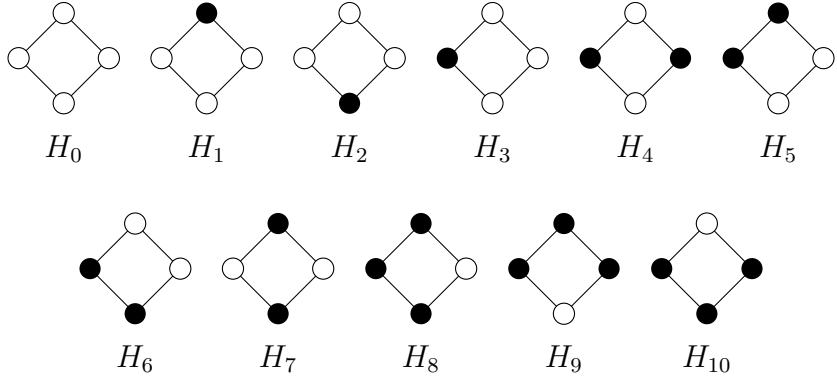
An additional constraint for the flags is that the labeled vertex is from $A_n$ or $C_n$, and the unlabeled vertex is from $B_n$. 
A black vertex indicates that the corresponding subset of $[n]$ is present in the subposet and  a white vertex indicates the opposite.

\begin{figure}
\begin{center}
{\def\e{0.8}
\def\c{ circle(3pt) }
\def\s{3}
\def\theta{\draw[fill=gray!30] (0,\e)  circle (6pt) ;}
\def\lab{-.6}

\def\confA{
\theta
\draw (0,\e)   -- (-\e,0);
\fill (0,\e)  \c   (-\e,0) \c  ;
\filldraw[fill=white]    (-\e,0) \c;
\draw (-0.5*\e,\lab) node {$F_0$} ;
}

\def\confB{
\theta
\draw (0,\e)   -- (-\e,0);
\fill (0,\e)  \c   (-\e,0) \c  ;
\draw (-0.5*\e,\lab) node {$F_1$} ;
}

\begin{tikzpicture}

\begin{scope}[xshift=0cm, yshift = 0cm] 
\confA 
\end{scope}
\begin{scope}[xshift=\s cm, yshift = 0cm] 
\confB
\end{scope}

\end{tikzpicture}}
\end{center}
\caption{Two flags with one labeled vertex.}
\label{fig-middF}
\end{figure}
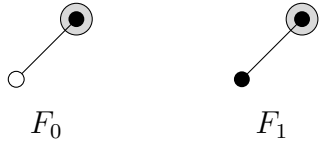

Given $G_n\in \cH_n$, let $p(H_i, G_n)$ be the probability that a random subset $D\simeq Q_2$ of $M_n$ chosen uniformly at random satisfies $D\cap G_n \simeq H_i$, then

$$\d(G_n) = \sum_{i} \d(H_i)p(H_i, G_n).$$

For the flags, for a vertex $\theta$ in $A_n\cup C_n$, we define $p(F_i,\theta,G_n)$ to be the probability that a random vertex $v$ from $B_n$ that is adjacent to $\theta$ (i.e. the set corresponding to $v$ contains the set corresponding to $\theta$ or is in $\theta$) satisfies $\{\theta,v\}\simeq F_i$. We also  define $p(F_i,F_j,\theta,G_n)$ to be the probability that two random vertices $u\ne v$ from $B_n$ that are adjacent to $\theta$ satisfy $\{\theta,u\}\simeq F_i$ and $\{\theta,v\}\simeq F_j$. 
A lemma analogous to Lemma~\ref{lem:productdensity} can be proven, we omit the details. Hence we can apply flag algebra method to this setup and get Table \ref{tab-M}.

\begin{table}[ht]
\begin{center}
\begin{tabular}{|c|c|c|c|c|c|c|c|c|c|c|c|}\hline 
         & $H_0$ & $H_1$ & $H_2$ & $H_3$ & $H_4$ & $H_5$ & $H_6$ & $H_7$ & $H_8$ & $H_9$ & $H_{10}$ \\ \hline 
             $\d$ & 0 & 1  & 1 & 1/2 & 1 & 3/2 & 3/2 & 2 & 5/2 & 2 & 2 \\\hline 
 $F_0, F_0$ & 0 & 1/2 & 1/2 & 0 & 0 & 0 & 0 & 1 & 0 & 0 & 0 \\\hline 
 $F_0, F_1$ & 0 & 0 & 0 & 0 & 0 & 1/4 & 1/4 & 0 & 1/2 & 0& 0  \\\hline 
 $F_1, F_1$ & 0 & 0 & 0 & 0 & 0 & 0 & 0 & 0 & 0 & 1/2 & 1/2\\\hline \end{tabular}  
\end{center}
\caption{$\d(H_k)$ and $\E_{\theta}p(F_i,F_j,\theta,H_k)$.}
\label{tab-M}
\end{table}

Then a semidefinite matrix 
$$M=
\left( \begin{array}{cc}
\frac{\sqrt{2}-1}{2} & \frac{\sqrt{2}-2}{2} \\ \frac{\sqrt{2}-2}{2} & \sqrt{2}-1
\end{array}\right )$$
gives the desired bound $\frac{3+\sqrt{2}}{2}$.

\begin{figure}
\begin{center}
{\input{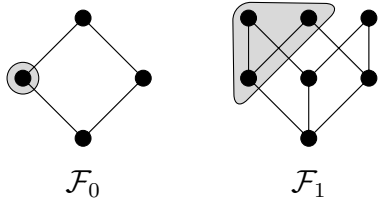}}
\end{center}
\caption{Flags families used in the computer assited proof.}
\label{fig-middF-comp}
\end{figure}

The proof of Theorem~\ref{thm:middle} goes along the same lines as for $\frac{3+\sqrt{2}}{2}$. One difference
is that three middle layers of $Q_4$ are considered instead of $Q_2$. The number of $Q_2$-free subgraphs
is 606.  The other difference is that we use flag families depicted in Figure~\ref{fig-middF-comp}.
Each family contains flags obtained from the depicted ones by coloring the vertices black and white.
Sources of a program for generating $Q_2$-free subgraphs and computing an analog of Table~\ref{tab-M}
are available at \oururl.

\section{Conclusion}

We presented an adaptation of Razborov's flag algebra method
to subgraphs of the hypercube. Using the adaptation we obtained
new upper bounds on densities in limit on 4-cycle and 6-cycle free subgraphs
of the hypercube. 

We suspect that the method can give a better bound when applied to
the hypercubes of dimension greater than 3. 
However, we found $3212821$ $C_4$-free spanning subgraphs of $\Q_4$. 
The resulting semidefinite program is currently too large for CSDP. 

We were trying to reduce the number of considered $C_4$-free subgraphs
by identifying those with the same $\d(H)+c_H$. The only set of flags we discovered
that was leading to a solvable semidefinite program was consisting of flags
whose vertices induce a star in the hypercube. 
See $F_1$ in Figure~\ref{fig-okflag} for an example. 
In this setting $\d(H_1)+c_{H_1} = \d(H_2)+c_{H_2}$ if $C_4$-free spanning subgraphs 
$H_1$ and $H_2$ have the same
degree sequence. Unfortunately, the resulting bounds were worse than the
bounds obtained from $\Q_3$ and square like flags.

Maybe a good set of flags, a better solver or just some future hardware can make such problems solvable.

\bibliographystyle{habbrv}
\bibliography{refs}

\end{document}